\documentclass[dvips,preprint]{imsart}

\usepackage{amsthm,amsmath}              %,natbib}
\usepackage[T1]{fontenc}
\usepackage[dvips]{graphics}
\usepackage{pifont}
\usepackage{amsfonts}
\usepackage{latexsym}
\usepackage{amsmath}
\usepackage{amssymb}
\usepackage{amscd}
\usepackage{color}
\usepackage{verbatim}
\usepackage{dsfont,bbold}
\usepackage{xfrac} 
\newcommand{\rouge}[1]{\textcolor{red}{#1}}

\renewcommand{\P}{{\mathds{P}}}
\newcommand{\E}{{\mathds{E}}}

\newcommand{\diag}{\mathrm{diag}}

\newcommand{\ud}{\, \mathrm{d}}  %  "d" for integrals
\newcommand{\ue}{\mathrm{e}}  %  "e" for exponentials

\newcommand{\vect}[1]{\boldsymbol #1}

\newcommand{\vone}{\vect 1}
\newcommand{\vzero}{\vect 0}

\newcommand{\epsi}{\varepsilon}

\newcommand{\vligne}[1]{\begin{bmatrix} #1 \end{bmatrix}}

\newcommand{\up}{\bs{p}_{+1}^{(b)}} 
\newcommand{\down}{\bs{p}_{-1}^{(0)}} 

\newcommand{\Mp}{\ue^{\Theta^{-1}\Psi_1^*b}}
\newcommand{\Mm}{\ue^{\Theta^{-1}\Psi_1b}}

\newtheorem{defn}{Definition}[section]
\newtheorem{lem}[defn]{Lemma}

\newtheorem{theo}[defn]{Theorem}

\newtheorem{cor}[defn]{Corollary}
\newtheorem{rem}[defn]{Remark}
\newtheorem{ass}[defn]{Assumption}

\newcommand{\debproof}{\begin{proof}}
\newcommand{\finproof}{\end{proof}}

\newcommand{\intrusion}[1]{\begin{list}{}%
{\leftmargin\parindent\rightmargin0pt\listparindent0pt\parsep1ex%
\labelwidth0pt\labelsep0pt}\item[]\small\sf$\spadesuit$~#1 \end{list}}
\newcommand{\bs}{\boldsymbol}

\newcommand{\barrer}[1]{\rouge{\sout{#1}}}
\renewcommand{\barrer}[1]{}

\begin{document}

 \begin{frontmatter}

 % "Title of the paper"
 \title{Fluid approach to two-sided Markov-modulated Brownian motion}
 \runtitle{Stationary distribution of two-sided MMBM}

 % indicate corresponding author with \corref{}
 % \author{\fnms{John} \snm{Smith}\corref{}\ead[label=e1]{smith@foo.com}\thanksref{t1}}
 % \thankstext{t1}{Thanks to somebody} 
 % \address{line 1\\ line 2\\ printead{e1}}
 % \affiliation{Some University}

 \author{\fnms{Guy} \snm{Latouche}\ead[label=e1]{guy.latouche@ulb.ac.be}
 }
 \address{Universit\'{e} libre de Bruxelles \\ 
 D\'{e}partement d'Informatique \\ \printead{e1}}
 % \and
 \author{\fnms{Giang T.} \snm{Nguyen}\corref{} \ead[label=e2]{giang.nguyen@adelaide.edu.au}\thanksref{t1}}
 \address{The University of Adelaide \\ School of Mathematical Sciences \\ \printead{e2}}
 \thankstext{t1}{The authors thank the Minist\`ere de la
 Communaut\'e fran\c{c}aise de Belgique for funding this research
 through the ARC grant AUWB-08/13--ULB~5.  They also acknowledge the
 financial support of the Australian Research Council through the
 Discovery Grant DP110101663.}

 % Remove this upon submission
 \vspace*{0.2cm}  

 \runauthor{G. Latouche and G. T. Nguyen}

 \begin{abstract}
   We extend to Markov-modulated Brownian motion (MMBM) the renewal
   approach which has been successfully applied to the analysis of
   Markov-modulated fluid models.  It has recently been shown that
   MMBM may be expressed as the limit of a parameterized family of
   Markov-modulated fluid models. We prove that the weak convergence also holds for systems with two reflecting boundaries, one at
   zero and one at $b >0$, and that the stationary distributions of the approximating fluid models converge to the stationary distribution of the two-sided reflected MMBM. Thus, we obtain a new representation for the
   stationary distribution, effectively separating the limiting
   behaviour of the process at the boundaries from its behaviour in
   the interior of $(0,b)$.
 \end{abstract}



 \end{frontmatter}

\section{Introduction} 
	\label{sec:intro}
Since the beginning of the last century, Brownian motions have been an
important class of stochastic processes, with applications in
increasingly diverse areas such as biology, queueing theory, physics,
environmental modeling, and mathematical finance. Naturally, the
effectiveness of Brownian motions as modeling tools has lead to their
many generalizations, one of which are the class of {Markov-modulated
  Brownian motions} (MMBMs), where the drift and variance are driven
by an independent, continuous-time finite-state Markov chain. Thus,
Markov-modulated Brownian motions are not only mathematically
fascinating but also applicable for a wide 
variety of real-life applications. 

Traditionally, the stationary distribution of MMBMs has been analyzed mainly via the theory of generators of Markov processes 
{in  Rogers}~\cite{roger94}, partial differential equations 
{in Karandikar and Kulkarni}~\cite{kk95}, the theory of martingales 
{in  Asmussen}~\cite{asmussen95} {and Asmussen and Kella}
\cite{ao00}, and generalized Jordan chains in D'Auria \emph{et al.}~\cite{dikm12}. Recently, there
appeared a fifth approach, via an approximation by Markov-modulated fluid
flows (MMFFs). First, Ramaswami~\cite{ram11} constructed a
parameterized family of MMFFs that converge weakly to a standard
Brownian motion. Then, Latouche and Nguyen~\cite{ln13} generalized
this construction to approximate MMBMs, with and without a reflecting
boundary at level zero{;} the authors showed that the
stationary distributions of approximating fluid processes converge to
the stationary distribution of the limiting reflected one-sided MMBM,
assuming that the latter process is positive recurrent.
 
%Here, our main contribution is a new proof for the stationary distribution of a reflected two-sided MMBM, that is, when the process is reflected not only at level $0$ but also at level $b > 0$. Our proof is significantly different from the ones obtained via the theory of generators~\cite{roger94} and via time-reversal arguments~\cite{ivanovs10}. Is this important enough to be done? Who knows. It is done anyway.

Here, we apply the fluid-based approximation approach to carry out the
stationary analysis for reflected two-sided Markov-modulated Brownian
motions, with boundaries at zero and at $b > 0$. This provides us with
a new representation for the stationary distribution, obtained via a
proof significantly different from the ones that rely on the theory of
generators~\cite{roger94} or on time-reversal
arguments~\cite{ivanovs10}. The new representation indicates that the
stationary density is a product of two terms, one of which is about
the limiting behaviour at the interior $(0,b)$ and the other is about
the limiting behaviour at the boundaries. This opens the way to the analysis of more
complex models.

In Section~\ref{sec:prel}, we formally define Markov-modulated
Brownian motions and Markov-modulated fluid flows, and describe the
fluid-based approximation in~\cite{ln13}. We show in
Section~\ref{sec:2s} that the stationary distributions of {the approximating processes} converge to the stationary distribution of the MMBM, and we determine
in Section~\ref{sec:form} the closed-form expression for the limiting
stationary distribution. In Section~\ref{sec:comp}, we draw comparison
between our representation and the ones derived in~\cite{roger94, ivanovs10}.

\section{Preliminaries} 
	\label{sec:prel} 
	
\subsection{Markov-modulated models} 
	
A Markov-modulated Brownian motion $\mathcal{Y} = \{Y(t), \kappa(t): t \geq 0\}$ is a continuous-time two-dimensional Markov-process, where the \emph{phase} $\kappa(t)$ is a Markov-chain on a finite state space $\mathcal{M} = \{1, \ldots, m\}$, the \emph{level} $Y(t) \in (-\infty, \infty)$ is a Brownian motion with drift $\mu_i$ and variance $\sigma^2_i$ whenever $\kappa(t) = i \in \mathcal{M}$. We denote by $D$ the drift matrix $\diag(\mu_1, \ldots, \mu_m)$, by $V$ the variance matrix $\diag(\sigma^2_1, \ldots, \sigma^2_m)$, and by $Q$ the generator of $\kappa(t)$, which we assume to be irreducible. We also assume the following. 

\begin{ass} 
	\label{ass:mmbm}
The initial level Y(0) is zero, the initial phase $\kappa(0)$ has the stationary distribution $\bs{\alpha}$ (that is, $\bs{\alpha}Q = \bs{0}$, $\bs{\alpha}\bs{1} = 1$), and $\sigma_i^2 \neq 0$ for all $i \in \mathcal{M}$. 
\end{ass}

Markov-modulated Brownian motions are sometimes referred to as \emph{second-order fluid models}; similarly, Markov-modulated fluid flows are also known as \emph{first-order fluid models}. A Markov-modulated fluid flow $\mathcal{L} = \{L(t), \varphi(t): t \geq 0\}$ is a continuous-time two-dimensional Markov process, where the phase $\varphi(t)$ is a Markov chain on a finite state space $\mathcal{S}$, the level $L(t) \in (-\infty, \infty)$ is independent of $\varphi(t)$ and 
% 
%\begin{align*} 
{
$\frac{\ud}{\ud t} L(t) = c_i$ {if } $\varphi(t) = i \in \mathcal{S}$. }
%\end{align*} 

\subsection{A fluid-based approximation}

Given {the} Markov-modulated Brownian motion $\mathcal{Y} =
\{Y(t),\kappa(t)\}$ {defined above,} and given Assumption~\ref{ass:mmbm}, we construct a parameterised family of fluid flows $\{L_{\lambda}(t), \beta_{\lambda}(t), \varphi_{\lambda}(t): t \geq 0\}$ as follows. The phase process here is a two-dimensional Markov chain $\{\beta_{\lambda}(t), \varphi_{\lambda}(t)\}$ on state space $\mathcal{S} = \{(k,i): k \in \{1, 2\} \mbox{ and } i \in \mathcal{M}\}$, with generator 
\begin{align*}
T_{\lambda} =  
\vligne{
Q - \lambda I & \lambda I \\ 
\lambda I & Q - \lambda I
},
\end{align*} 
where {the} components of $T_{\lambda}$ are indexed according to lexicographic ordering of $\{1,2\} \times \mathcal{M}$, the parameter $\lambda$ is positive, and $I$ denotes the identity matrix of appropriate dimensions. Whenever ambiguity might arise, we write $I_n$ to denote the $n \times n$ identity matrix. The rate matrix $C_{\lambda} = \diag(c_{k,i})_{k \in \{1,2\}, i \in \mathcal{M}}$ for the level $L_{\lambda}(t)$ is given by 
\begin{align*} 
C_{\lambda} = 
\vligne{
D + \sqrt{\lambda}\Theta & \\ 
& D - \sqrt{\lambda}\Theta 
}, \quad \mbox{ where } \Theta = \sqrt{V}. 
\end{align*} 
As $Q$ is by assumption irreducible, so is $T_{\lambda}$, and for sufficiently large values of~$\lambda$ the matrix $C_{\lambda}$ is invertible. 

\begin{ass}
	\label{ass:fluid} 
The initial level $L_{\lambda}(0)$ is zero, $\beta_{\lambda}(0)$ has the stationary distribution $\bs{\gamma} = (\sfrac{1}{2}\;, \sfrac{1}{2})$, and $\varphi_{\lambda}(0)$ has the stationary distribution $\bs{\alpha}$.  
\end{ass} 

{Informally,} we duplicate  {for the constructed fluid
  model} the state space~$\mathcal{M}$ of the phase process
$\kappa(t)$, and keep track of each copy via $\beta_{\lambda}(t) \in
\{1,2\}$. The process $\{\beta_{\lambda}(t), \kappa_{\lambda}(t)\}$
switches from a phase in a copy (say, $(1,i)$) to the corresponding
phase in the other copy (which would be $(2,i)$) with rate $\lambda$;
the dynamic between phases in a copy is the same as that of the phase
process $\kappa(t)$, governed by $Q$. As $\lambda$ tends to infinity,
$(\beta_{\lambda}(t), \kappa_{\lambda}(t))$ switches between two
corresponding phases faster and faster, effectively scaling time. The
matrix $C_{\lambda}$ implies, that for the duplicated phases we modify
the original drifts by an increment of $\sqrt{\lambda} \Theta$ for one
copy and by a decrement of the same quantity for the other copy. As
$\lambda$ tends to infinity, so does the difference between two drifts
of corresponding phases, effectively scaling space. 
The {combined} scaling of space and of time is the underlying reason for the convergence of parameterised fluid flows to the Markov-modulated Brownian motion $\{Y(t), \kappa(t)\}$. This weak convergence, proved in~\cite{ln13}, is formally stated in the following theorem. 
\begin{theo}[\cite{ln13}]
	\label{theo:conv} 
	Given Assumptions~\ref{ass:mmbm} and~\ref{ass:fluid}, the processes $\{L_{\lambda}(t), \varphi_{\lambda}(t): t \geq 0\}$ converge weakly to $\{Y(t), \kappa(t): t \geq 0\}$, as $\lambda \rightarrow \infty$. 
\end{theo} 
\subsection{Reflected one-sided processes} 
	
We can construct a similar approximation by fluid flows for MMBMs with a reflecting boundary at level zero. Denote by $\widehat{\mathcal{Y}} = \{\widehat{Y}(t), \kappa(t): t \geq 0\}$ the reflected one-sided MMBM associated with $\{Y(t), \kappa(t): t \geq 0\}$, where 
\begin{align*}
\widehat{Y}(t) = Y(t) - \inf_{0 \leq v \leq t}Y(v),  
\end{align*}
and by $\widehat{\mathcal{L}}_{\lambda} = \{\widehat{L}_{\lambda}(t), \beta_{\lambda}(t), \varphi_{\lambda}(t): t \geq 0\}$ the resulting fluid flow if we introduce into $\{L_{\lambda}(t), \beta_{\lambda}(t), \varphi_{\lambda}(t)\}$ a reflecting boundary at zero: 
\begin{align*}
\widehat{L}_{\lambda}(t) = L_{\lambda}(t) - \inf_{0 \leq v \leq t} L_{\lambda}(v).  
\end{align*} 
By applying the one-sided reflection map to $Y(t)$ and to $L_{\lambda}(t)$, we know that the process $\widehat{Y}(t)$ exists uniquely and so does $\widehat{L}_{\lambda}(t)$. The following result immediately follows from Theorem~\ref{theo:conv}. 

\begin{cor}[\cite{ln13}] 
The processes $\{\widehat{L}_{\lambda}(t), \varphi_{\lambda}(t): t \geq 0\}$ weakly converge, as $\lambda \rightarrow \infty$, to the reflected one-sided MMBM $\{\widehat{Y}(t), \kappa(t): t \geq 0\}$. 
\end{cor} 

If the process $\{\widehat{Y}(t), \kappa(t)\}$ is positive recurrent,
{that is, if} $\bs{\alpha}D\bs{1} < 0$, then the two limits, of
$\lambda$ and of $t$, are interchangeable, {and} the limiting distribution of $\{\widehat{L}_{\lambda}(t), \varphi_{\lambda}(t)\}$ converge, as $\lambda \rightarrow \infty$, to the joint stationary distribution of $\{\widehat{Y}(t), \kappa(t)\}$ \cite[Theorem~3.6]{ln13}. 

\section{Reflected two-sided Markov-modulated Brownian motions}
	\label{sec:2s}
	
Here, we consider processes with not only a reflecting boundary at level zero but also one at level $b$, for some finite $b > 0$. Let $\widetilde{\mathcal{Y}} = \{\widetilde{Y}(t),\kappa(t): t \geq 0\}$ be the reflected two-sided MMBM associated with $\{Y(t), \kappa(t): t \geq 0\}$, where 
\begin{align*}
\widetilde{Y}(t) = Y(t) + W(t) - M(t) \in [0,b],  
\end{align*} 
with $W(t)$ and $M(t)$ being the local times at level zero and level $b > 0$, respectively. More specifically, $W(t)$ and $M(t)$ are processes that satisfy the following conditions: $W(t)$ and $M(t)$ are nondecreasing with $W(0) = M(0) = 0$; $\widetilde{Y}(s) = 0$ if $W(s) < W(t)$ for all $t > s$; and $\widetilde{Y}(s) = b$ if $M(s) < M(t)$ of all $t > s$. By applying the two-sided reflection map on $[0,b]$ (Kruk \emph{et al.}~\cite{klrs07}), we obtain existence and uniqueness for $\widetilde{Y}(t)$.

We denote by $\widetilde{\mathcal{L}}_{\lambda} = \{\widetilde{L}_{\lambda}(t), \beta_{\lambda}(t), \varphi_{\lambda}(t): t \geq 0\}$ the finite-buffer fluid process associated with the unbounded process $\{L_{\lambda}(t), \beta_{\lambda}(t), \varphi_{\lambda}(t)\}$, where for $\beta_{\lambda}(t) = k \in \{1,2\}$ and $\varphi_{\lambda}(t) = i \in \mathcal{M}$
\begin{align*}
\begin{array}{rlll} 
\displaystyle\frac{\ud}{\ud t} \widetilde{L}_{\lambda}(t) = & \hspace*{-0.2cm} c_{k,i} & \mbox{ if } \widetilde{L}_{\lambda}(t) \in (0,b), \\
                                                                  = & \hspace*{-0.2cm} \max\{0,c_{k,i}\} & \mbox{ if } \widetilde{L}_{\lambda}(t) = 0, \\
                                                                  = & \hspace*{-0.2cm} \min \{0,c_{k,i}\} & \mbox{ if } \widetilde{L}_{\lambda}(t) = b.                                                    
\end{array} 
\end{align*}
In other words, in between the boundaries at $0$ and at $b$ the process $\widetilde{L}_{\lambda}(t)$ evolves the same way $L_{\lambda}(t)$ does. Upon hitting level $0$ (or level $b$), $\widetilde{L}_{\lambda}(t)$ remains there until the phase process $\{\beta_{\lambda}(t), \varphi_{\lambda}(t)\}$ switches to a phase with positive rate (or, respectively, negative rate). The process $\widetilde{L}_{\lambda}(t)$ can be obtained by applying the two-sided reflection map on $[0,b]$ to $L_{\lambda}(t)$, and therefore exists uniquely. The following corollary follows immediately from Theorem~\ref{theo:conv}.
\begin{cor}
	\label{cor:2s-conv} 
	The processes $\{\widetilde{L}_{\lambda}(t), \varphi_{\lambda}(t): t \geq 0\}$ weakly converge to the reflected two-sided MMBM $\{\widetilde{Y}(t), \kappa(t): t \geq 0\}$, as $\lambda \rightarrow \infty$.
\end{cor} 

\begin{ass}
	\label{ass:drift} 
{The mean drift $\bs{\alpha} D \bs{1}$ of $\{Y(t), \kappa(t)\}$
  is different from zero.}
\end{ass} 

{The mean drift
of $\{\widetilde{L}_{\lambda}(t), \varphi_{\lambda}(t)\}$ is
 $\bs{\gamma}\hspace*{-0.05cm}\otimes\hspace*{-0.05cm}\bs{\alpha}
 C_{\lambda} \bs{1}$ and it is straightforward to verify that  
$\bs{\gamma}\hspace*{-0.05cm}\otimes \hspace*{-0.05cm}\bs{\alpha}
C_{\lambda} \bs{1} = \bs{\alpha} D \bs{1}$ independently of $\lambda$.}

In order to determine that the joint stationary distributions of approximating one-sided fluids $\{\widehat{L}_{\lambda}(t),\varphi_{\lambda}(t)\}$ converge to that of the one-sided MMBM $\{\widehat{Y}(t), \kappa(t)\}$, Latouche and Nguyen~\cite[Theorem~3.6]{ln13} show that the limiting distribution is equivalent to the stationary distribution of $\{\widehat{Y}(t), \kappa(t)\}$ obtained in Asmussen~\cite{asmussen95}. Here, we follow a more direct approach to show convergence of stationary distributions of two-sided processes. 

Denote by $\bs{\widetilde{F}}_{\lambda}(x)$ the joint stationary distribution vector of $\{\widetilde{L}_{\lambda}(t), \varphi_{\lambda}(t)\}$, with components 
\begin{align}
[\bs{\widetilde{F}}_{\lambda}(x)]_{i} & = \lim_{t \rightarrow \infty} \P[\widetilde{L}_{\lambda}(t) \leq x, \varphi_{\lambda}(t) = i] \quad \mbox{for } x \in [0,b] \mbox{ and } i \in \mathcal{M}. 
%\end{align*} 
%
\intertext{and by $\bs{\widetilde{F}}(x)$ its element-wise limit, where}
% 
%\begin{align*}
	\label{eqn:limFx} 
[\bs{\widetilde{F}}(x)]_i & = \lim_{\lambda \rightarrow \infty}[\bs{\widetilde{F}}_{\lambda}(x)]_i \quad \mbox{for } x \in [0,b] \mbox{ and } i \in \mathcal{M}. 
\end{align} 
We prove in the next section that the limit $\bs{\widetilde{F}}(x)$ defined in~\eqref{eqn:limFx} exists. Here, to preserve the flow we assume its existence and show that this limit is indeed the stationary distribution $\bs{\widetilde{G}}(x)$ of $\{\widetilde{Y}(t), \kappa(t)\}$, where 
\begin{align*}
 [\bs{\widetilde{G}}(x)]_{i} = \lim_{t \rightarrow \infty} \P[\widetilde{Y}(t) \leq x, \kappa(t) = i] \quad \mbox{for } x \in [0,b] \mbox{ and } i \in \mathcal{M}. 
\end{align*} 

First, we extend Theorem~\ref{theo:conv} by modifying its assumptions that $L_{\lambda}(0) = 0$ and $Y(0) = 0$. 
% 
% \begin{ass}
% 	\label{ass:newmmbm} 
% 	$\{Y(0),\kappa(0)\}$ has the distribution $\bs{\widetilde{F}}$. 
% \end{ass} 

% \begin{ass} 
% 	\label{ass:newfluid} 
% 	$\{L_{\lambda}(0),\varphi_{\lambda}(0)\}$ has the distribution $\bs{\widetilde{F}}_{\lambda}$. 
% \end{ass} 

\begin{theo}
  \label{theo:conv-newa} 
Assume that $(Y(0), \kappa(0))$ has the
    distribution $\bs{\widetilde{F}}$ and that $(L_{\lambda}(0),
    \varphi_{\lambda}(0))$ has the distribution
    $\bs{\widetilde{F}}_{\lambda}$.
The family of processes $\{L_{\lambda}(t), \varphi_{\lambda}(t): t \geq 0\}$ converge to the Markov-modulated Brownian motion $\{Y(t),\kappa(t): t \geq 0\}$, as $\lambda \rightarrow \infty$.  
\end{theo}
\begin{proof} 
First, we prove that the finite-dimensional
distributions of $\{L_{\lambda}(t), \varphi_{\lambda}(t)\}$ converge
to those of $\{Y(t), \kappa(t)\}$ via convergence of moment generating
functions, that is, {we show that }
\begin{align}
	\label{eqn:limcmgf}
        \lim_{\lambda \rightarrow \infty}
        \E[\ue^{sL_{\lambda}(t)}\mathds{1}_{\{\varphi_{\lambda}(t) =
          j\}} | (L_{\lambda}(0),  \varphi_{\lambda}(0))=_d
        \bs{\widetilde{F}}_{\lambda}] \nonumber \\ 
        = \E[\ue^{sY(t)}\mathds{1}_{\{\kappa(t) = j\}} | (Y(0),
        \kappa (0)) =_d \bs{\widetilde{F}}].
\end{align} 
The marginal stationary distribution of the phase
$\varphi_\lambda$ is $\bs{\alpha}$, and we may write that
\begin{align}
\E[\ue^{sL_{\lambda}(t)} & \mathds{1}_{\{\varphi_{\lambda}(t) = j\}} |
(L_{\lambda}(0), \varphi_{\lambda}(0)) =_d  {\bs{\widetilde{F}}_{\lambda}}]
\nonumber \\[0.5\baselineskip]
& = \sum_{i \in \cal M} \alpha_i \E[\ue^{s L_{\lambda}(t)}
  \mathds{1}_{\{\varphi_{\lambda}(t) = j\}} | 
L_{\lambda}(0) =_d \bs{\widetilde{F}}_{\lambda | i},
\varphi_{\lambda}(0) = i] \nonumber 
\intertext{where $\bs{\widetilde{F}}_{\lambda | i}$ is the conditional
  stationary distribution of $\widetilde L_\lambda$, given that the
  phase is $i$,}
& = \sum_{i \in \mathcal{M}} \alpha_i \E[\ue^{s(L_{\lambda}(t) +
  \widetilde{F}_{\lambda | i})}\mathds{1}_{\{\varphi_{\lambda}(t) = j\}}
| L_{\lambda}(0) = 0, \varphi_{\lambda}(0) = i] \nonumber
\intertext{where $\widetilde{F}_{\lambda | i}$ is a random variable with
  distribution $\bs{\widetilde{F}}_{\lambda | i}$,}
& = \sum_{i \in \mathcal{M}} \alpha_i
\E[\ue^{s\widetilde{F}_{\lambda | i}}]
\E[\ue^{sL_{\lambda}(t)}\mathds{1}_{\{\varphi_{\lambda}(t) = j\}} |
L_{\lambda}(0) = 0, \varphi_{\lambda}(0) = i] \nonumber  \\
& = \bs{\alpha} \Gamma_\lambda(s)\ue^{\Delta_{\lambda}(s)t}\bs{e}_j 
\label{eqn:cmgfL}
\end{align}
where $\Gamma_\lambda(s)$ is a diagonal matrix with
$\E[\ue^{s\widetilde{F}_{\lambda | i}}] $, $i \in \cal M$, on the
diagonal, $\bs{e}_j$ is an $m \times 1$ vector with zeros in all
entries except the $j$th one, and $\Delta_{\lambda}(s)$ is the Laplace
matrix exponent of $\{L_{\lambda}(t), \varphi_{\lambda}(t)\}$ which satisfies
\begin{align*} 
[\ue^{\Delta_{\lambda}(s)}]_{ij} = \E[\ue^{sL_{\lambda}(t)}\mathds{1}_{\{\varphi_{\lambda}(t) = j\}}| L_{\lambda}(0) = 0, \varphi_{\lambda}(0) = i]
\end{align*} 
{for $i$, $j$ in ${\cal M}$.}   Similarly, 
\begin{align} 
 \E[\ue^{sY(t)} \mathds{1}_{\{\kappa(t) = j\}} | (Y(0), \kappa(0)) =_d
{\bs{\widetilde{F}}}]
 = \bs{\alpha} \Gamma(s)\ue^{\Delta_{Y}(s)t}\bs{e}_j,  
\label{eqn:cmgfY}
\end{align} 
where $\Gamma(s)$, by the definition of $\bs{\widetilde{F}}$, is the
limit of $\Gamma_\lambda(s)$ as $\lambda \rightarrow \infty$ and 
$\Delta_Y(s)$ is the Laplace matrix exponent of $\{Y(t), \kappa(t)\}$.
 In addition, 
Theorem~2.4 in~\cite{ln13} states
that $\lim_{\lambda \rightarrow \infty} \ue^{\Delta_{\lambda}(s)t} =
\ue^{\Delta_{Y}(s)t}$, and so (\ref{eqn:cmgfL},  \ref{eqn:cmgfY})
imply~\eqref{eqn:limcmgf}.

Now, 
we show that the family $\{L_{\lambda}(t), \beta_{\lambda}(t), \varphi_{\lambda}(t)\}$ is still tight under the new initial condition. By Theorem~8.3~in Billingsley~\cite{billingsley68} and Whitt~\cite{whitt70}, it is sufficient to verify the following two conditions 
\begin{enumerate} 
\item[(i)] for each $\eta > 0$, there exists $a$ such that 
\begin{align} 
\P[L_{\lambda}(0) > a] \leq \eta \quad \mbox{ for sufficiently large } \lambda, \nonumber
\end{align} 
\item[(ii)] for each $\varepsilon, \eta > 0$, there exists $\delta \in (0,1)$ and $\lambda_0$ such that 
\begin{align} \label{eqn:aim} 
\frac{1}{\delta} \P[\sup_{t \leq s \leq t + \delta} |L_{\lambda}(s) - L_{\lambda}(t)| \geq \varepsilon] \leq \eta \quad \mbox{ for all $\lambda \geq \lambda_0$ and $t > 0$.}
\end{align} 
\end{enumerate}
Condition~(ii) follows from the proof of Theorem~2.7 in~\cite{ln13}, which show that the family $\{L_{\lambda}(t), \beta_{\lambda}(t), \varphi_{\lambda}(t)\}$ is tight given that $L_{\lambda}(0) = 0$ and $\varphi_{\lambda}(0) =_d \bs{\alpha}$. Condition~(i) is immediately satisfied by setting $a = b$, the upper reflecting boundary. 
\end{proof} 
We are now ready to show that the two limits, of $\lambda$ and of $t$, are also interchangeable in the two-sided case. In other words, 
\begin{theo} 
	\label{theo:inter}
The limiting distribution of $\{\widetilde{L}_{\lambda}(t), \varphi_{\lambda}(t): t \geq 0\}$ converges, as $\lambda \rightarrow \infty$, to the stationary distribution of $\{\widetilde{Y}(t), \kappa(t): t \geq 0\}$. 
\end{theo} 
\begin{proof} 
Let $\{0 = t_0 \leq t_1 \leq t_2 \leq \cdots\}$ be a sequence of {arbitrary} time epochs. 
{Assume that $\{L_{\lambda}(0),\varphi_{\lambda}(0)\}$ has the
  distribution $\bs{\widetilde{F}}_{\lambda}$.  For
} $x \in [0,b]$ and $i \in \mathcal{M}$ 
\begin{align}
	\label{eqn:Ltk} 
\P[\widetilde{L}_{\lambda}(t_k) \leq x, \varphi_{\lambda}(t_k) = i] & = \P[\widetilde{L}_{\lambda}(0) \leq x, \varphi_{\lambda}(0) = i] \quad \mbox{ for all } k \geq 0.  
\end{align} 
On the other hand, Theorem~\ref{theo:conv-newa} implies that for $x \in [0,b]$ and $i \in \mathcal{M}$,
\begin{align}
%	\label{eqn:Ytk} 
\P[\widetilde{Y}(t_k) \leq x, \kappa(t_k) = i] & = \lim_{\lambda \rightarrow \infty} \P[\widetilde{L}_{\lambda}(t_k) \leq x, \varphi_{\lambda}(t_k) = i]  \nonumber \\
& = \lim_{\lambda \rightarrow \infty} \P[\widetilde{L}_{\lambda}(0) \leq x, \varphi_{\lambda}(0) = i]  \nonumber \\
& = [\bs{F}(x)]_i \nonumber 
% \\ 
% & = \P[\widetilde{Y}(0) \leq x, \kappa(0) = i], 
\end{align} 
independently of $t_k$. Thus, $\bs{F}$ is the stationary distribution of $\{\widetilde{Y}(t), \kappa(t)\}$. 
\end{proof}

% \intrusion{Explain why $\{\widetilde{Y}(t), \kappa(t)\}$ has a unique stationary distribution}

\section{Stationary distribution of two-sided MMBM} 
	\label{sec:form}  

In light of Theorem~\ref{theo:inter}, a key component for obtaining the stationary distribution of the two-sided MMMBM $\{\widetilde{Y}(t), \kappa(t)\}$ is the stationary distribution of the finite-buffer fluid process $\{\widetilde{L}_{\lambda}(t), \beta_{\lambda}(t), \varphi(t)\}$. 

For $k = 1, 2$ and $i \in \mathcal{M}$, let   
\begin{align*} 
{\pi}^{(b)}_{k,i}(x) & = \lim_{t \rightarrow \infty}  \frac{\ud}{\ud
  x} \P[\widetilde{L}{_\lambda}(t) \leq x, \beta_{\lambda}(t) =
k, \varphi_{\lambda}(t) = i], 
\qquad \mbox{{for $0 < x < b$,}}
\\ 
p_{k,i}^{(0)} & = \lim_{t \rightarrow \infty} \frac{\ud}{\ud x} \P[\widetilde{L}{_\lambda} (t) = 0, \beta_{\lambda}(t) = k, \varphi_{\lambda}(t) = i], \\
p_{k,i}^{(b)} & = \lim_{t \rightarrow \infty} \frac{\ud}{\ud x} \P[\widetilde{L}{_\lambda} (t) = b, \beta_{\lambda}(t) = k, \varphi_{\lambda}(t) = i] 
\end{align*} 
be the stationary density function and probability masses at two 
boundaries of $\{\widetilde{L}{_\lambda} (t), \beta_{\lambda}(t),
\varphi_{\lambda}(t)\}$, respectively. Also, define the stationary
density vector  
\begin{align*} 
\bs{\pi}^{(b)}(x) = (\pi_{k,i})_{k \in \{1,2\}, i \in \mathcal{M}} =
\vligne{\bs{\pi}^{(b)}_{+}(x) & \bs{\pi}^{(b)}_{-}(x)},
\end{align*} 
and the stationary probability mass vectors 
\begin{align*} 
\bs{p}^{(0)} = (p_{k,i}^{(0)})_{k \in \{1,2\}, i \in \mathcal{M}} \quad \mbox{ and } \quad \bs{p}^{(b)} = (p_{k,i}^{(b)})_{k \in \{1,2\}, i \in \mathcal{M}}.
\end{align*}
By their physical interpretations, $\bs{p}^{(0)} =
(\bs{0},\bs{p}_{-}^{(0)})$ and $\bs{p}^{(b)} =
(\bs{p}_{+}^{(b)},\bs{0})$. Da Silva Soares and
Latouche~\cite[Theorems~4.4 and 5.1]{dL05} give a representation for
the stationary density and probability masses at boundaries of a
finite-buffer fluid model, given that the rates of the fluid level are
restricted to $\pm$1. We extend their results to the case with general
rates, {for} which {we} require some notation and definitions. 

Let us partition the generator matrix $T_{\lambda}$ and the rate matrix $C_{\lambda}$ according to phases with positive and negative rates as follows
\begin{align*} 
T_{\lambda} = \vligne{ T_{++} & T_{+-} \\ T_{-+} & T_{--} } \quad \mbox{ and } \quad C_{\lambda} = \vligne{ C_{+} & \\ & C_{-} }. 
\end{align*} 
For notational convenience when dealing with expansion of infinite series later, we write $\lambda = 1/\varepsilon^2$. Next, define the matrices 
\begin{align*}
U_{\varepsilon} & = |C_{-}|^{-1}  T_{--} + |C_{-}|^{-1}T_{-+} \Psi_{\varepsilon}, \\
{U}^*_{\varepsilon} & = C_{+}^{-1}T_{++} + C_{+}^{-1}T_{+-}{\Psi}^*_{\varepsilon}, \\
K_{\varepsilon} & = C_{+}^{-1}T_{++} + \Psi_{\varepsilon}|C_{-}|^{-1}T_{-+}, \\
{K}^*_{\varepsilon} & = |C_{-}|^{-1}T_{--} + {\Psi}^*_{\varepsilon}C_{+}^{-1}T_{+-},
\end{align*} 
where $\Psi_{\varepsilon}$ is the minimal nonnegative solution to the Riccati equation 
\begin{align*}
C_{+}^{-1} T_{+-} + C_+^{-1}T_{++}\Psi_{\varepsilon} + \Psi_{\varepsilon} |C_{-}|^{-1} T_{--} + \Psi_{\varepsilon} |C_{-}|^{-1} T_{-+}\Psi_{\varepsilon} = 0,   
\end{align*} 
and ${\Psi}^*_{\varepsilon}$ is the minimal nonnegative solution to the Riccati equation 
\begin{align*} 
|C_{-}|^{-1} T_{-+} + |C_{-}|^{-1}T_{--}{\Psi}^*_{\varepsilon} + {\Psi}^*_{\varepsilon} C_{+}^{-1} T_{++} + {\Psi}^*_{\varepsilon} C_{+}^{-1} T_{+-}{\Psi}^*_{\varepsilon} = 0.   
\end{align*} 
It is well-known that $\Psi_\varepsilon$ and ${\Psi}^*_{\varepsilon}$
have probabilistic interpretations: $\Psi_{\varepsilon}$ records
return probabilities from above to initial level in the boundary-free
fluid process $\{{L}_{\lambda}(t), \varphi_{\lambda}(t)\}$, and
${\Psi}^*_{\varepsilon}$ records return probabilities from below to
initial level. 

\begin{lem}[{\cite{ln13}}] 
   \label{t:psi}
\begin{align}
	\label{eqn:Psie} 
\Psi_{\varepsilon} & = I + \varepsilon \Psi_1 + O(\varepsilon^2), \\
	\label{eqn:Psies}
\Psi_{\varepsilon}^* & = I + \varepsilon \Psi_1^* + O(\varepsilon^2), 
\end{align} 
where $\Theta^{-1}\Psi_1$ and $-\Theta^{-1}\Psi_1^*$ are solutions to the matrix quadratic equation $$\frac{1}{2}VX^2 + DX + Q = 0,$$ such that 
\begin{enumerate} 
\item[(i)] if $\bs{\alpha} D \bs{1} > 0$, $\Theta^{-1}\Psi_1$ has one eigenvalue equal to zero and $m - 1$ eigenvalues with strictly negative real part, it is the unique such solution; $-\Theta^{-}\Psi_1^*$ has $m$ eigenvalues with strictly positive real parts, it, too, is the unique such solution,
\item[(ii)] if $\bs{\alpha} D \bs{1} < 0$, $\Theta^{-1}\Psi_1$ has $m$ eigenvalues with strictly negative part, and $-\Theta^{-1}\Psi_1^*$ has $m - 1$ eigenvalues with strictly positive real parts and one eigenvalue equal to zero; both $\Psi^{-1}\Psi_1$ and $-\Theta^{-1}\Psi_1^*$ are still unique such solutions.
\end{enumerate} 
\end{lem} 

\begin{rem} \rm Lemma~3.4 in~\cite{ln13} gives the eigenvalue 
  characterisations of $\Theta^{-1}\Psi_1$ and $-\Theta^{-1}\Psi_1^*$
  in the case when the mean drift $\bs{\alpha} D \bs{1} $ is
  negative. We employ analogous reasoning to extend the
  results to the case $\bs{\alpha} D \bs{1} > 0$.
\end{rem}

\begin{rem} \rm{Let $\tau^{\pm}_x = \inf \{t {> 0}: \pm Y(t) > x\}$ be the first passage times to the corresponding levels $x$ and $-x$ of the unbounded process $Y(t)$. Under the assumption that $\sigma_i > 0$ for all $i \in \mathcal{M}$, it is easy to confirm that $\Theta^{-1}\Psi_1$ and $\Theta^{-1}\Psi_1^*$ are the same as, respectively, the generators $\Lambda^{-}$ and $\Lambda^{+}$ of the time-changed processes $\kappa(\tau_x^-)$ and $\kappa(\tau_x^+)$ in Ivanovs~\cite{ivanovs10}, and $\Theta^{-1}\Psi_1^*$ is the same as the matrix $U(\gamma)$ for $\gamma = 0$ in Breuer~\cite{breuer08}.}
\end{rem}

Now we are ready to express the stationary density for the general case. 
\begin{theo}                         %%% the old guy
	\label{theo:finitefluid}
		The stationary density vector $\bs{\pi}^{(b)}(x)$, for $0 < x < b$, of the finite-buffer fluid process $\{\widetilde{L}_{\lambda}(t), \beta_{\lambda}(t), \varphi_{\lambda}(t)\}$ is given by  
	\begin{align}
   \label{e:oldpi}
	\bs{\pi}_{\varepsilon}^{(b)}(x) & = \vect y
\vligne{ 
	\ue^{K_{\varepsilon}x} & 0  \\
	0 & \ue^{{K}^*_{\varepsilon}(b - x)} }
	\left[\begin{array}{rr}
	C_{+}^{-1} & \Psi_{\varepsilon} C_{-}^{-1}\\
	 \Psi^*_{\varepsilon} C_{+}^{-1} & |C_{-}|^{-1}
	 \end{array}\right]
	   \end{align}
	where
\begin{align} 
   \label{e:y}
  \vect y = \vligne{\bs{y}_{+} & \bs{y}_{-}} & =
  \vligne{\bs{p}_{+}^{(b)} & \bs{p}_{-}^{(0)}}
	\vligne{	0 & T_{+-} \\     T_{-+} & 0 	}
    N^{-1}
\end{align} 
and
\[
{N} =
	\vligne{ 
	I & \ue^{K_{\varepsilon}b}\Psi_{\varepsilon} \\ 
	\ue^{{K}^*_{\varepsilon}b}{\Psi}^*_{\varepsilon} & I 
	} .
\]
The boundary probability masses $\bs{p}_{+}^{(b)}$, $\bs{p}_{-}^{(0)}$ satisfy the system of equations 
\begin{align}
	\label{eqn:pms-1} 
& \vligne{\bs{p}^{(b)}_{+} & \bs{p}^{(0)}_{-}} W_{\varepsilon} = 0, \\
	\label{eqn:pms-2} 
& \vligne{\bs{p}_{+}^{(b)} & \bs{p}_{-}^{(0)}}\bs{1} + \int_0^b
\vligne{\bs{\pi}_{+}^{(b)}(x) & \bs{\pi}_{-}^{(b)}(x)} \bs{1} \ud x  =
  1,  
\end{align} 
with 
\begin{equation}
W_{\varepsilon} = 
\vligne{
T_{++} & 0\\
0 & T_{--} 
}
+ \vligne{
0 & T_{+-} \\
T_{-+} & 0
}
G^{(b)},
\end{equation}
where the matrix $G^{(b)}$ defined as 
\begin{equation}
G^{(b)} = 
\vligne{
\Lambda_{++}^{(b)} & \Psi_{+-}^{(b)} \\
\widetilde{\Psi}_{-+}^{(b)} & \widetilde{\Lambda}_{--}^{(b)} }
\end{equation}
is the solution of the system 
\begin{align}
	\label{eqn:probmats} 
\vligne{
\Lambda_{++}^{(b)} & \Psi_{+-}^{(b)} \\
\vspace*{-0.3cm} \\
\widetilde{\Psi}_{-+}^{(b)} & \widetilde{\Lambda}_{--}^{(b)}
}
\vligne{
I & \Psi_{\varepsilon}\ue^{U_{\varepsilon}b} \\
{\Psi}^*_{\varepsilon}\ue^{{U}^*_{\varepsilon}b} & I 
}
= \vligne{\ue^{{U}^*_{\varepsilon}b} & \Psi_{\varepsilon} \\
{\Psi}^*_{\varepsilon} & \ue^{U_{\varepsilon}b}}. 
\end{align} 
\end{theo} 

\begin{proof}  This is shown by adapting the proof of \cite[Theorem
  4.4]{dL05} to the general case where the fluid rates may be
  different from 1 or $-1$.  
\end{proof} 

We give below another expression for the stationary distribution,
which will be more convenient in the sequel.

\begin{cor}                            %%% the new guy
	\label{theo:finitefluidalt}
        The stationary density vector $\bs{\pi}{_\epsi}^{(b)}(x)$, for $0 < x
        < b$, and the probability masses $\bs{p}_{+}^{(b)}$ and
        $\bs{p}_{-}^{(0)}$ may also be written as
\begin{align}
   \label{e:piepsilon}
	\bs{\pi}_{\varepsilon}^{(b)}(x) & = c \vligne{\bs{\nu}_{+} &
        \bs{\nu}_{-}} {N}^{-1}
        \vligne{ 
	\ue^{K_{\varepsilon}x} & 0  \\
	0 & \ue^{{K}^*_{\varepsilon}(b - x)} }
	\left[\begin{array}{rr}
	C_{+}^{-1} & \Psi_{\varepsilon} |C_{-}|^{-1}\\
	 \Psi^*_{\varepsilon} C_{+}^{-1} & |C_{-}|^{-1}
	 \end{array}\right]
\intertext{and} 
   \label{e:pms}
\vligne{\bs{p}_{+}^{(b)} & \bs{p}_{-}^{(0)}} & = c \vligne{\bs{\nu}_{+} &
        \bs{\nu}_{-}} G^{(b)} \vligne{-T_{++}^{-1} & 0 \\ 0 & -T_{--}^{-1}},
\end{align}
	where the vector $\bs\nu = \vligne{\bs{\nu}_{+} & \bs{\nu}_{-}}$
is the stationary probability vector of the matrix 
\begin{equation}
   \label{e:H}
H = G^{(b)} \vligne{-T_{++}^{-1} & 0 \\ 0 & -T_{--}^{-1}}
\vligne{0 & T_{+-} \\ T_{-+} & 0}.
\end{equation}
and the scalar $c$ is the normalizing constant defined by
\begin{align}
% 	\label{eqn:pms-1} 
% (\bs{p}^{(b)}_{+}, \bs{p}^{(0)}_{-})W_{\varepsilon} & = 0, \\
	\label{eqn:pms-2b} 
\vligne{\bs{p}_{+}^{(b)} & \bs{p}_{-}^{(0)}}\bs{1} + \int_0^b
\vligne{\bs{\pi}_{+}^{(b)}(x) & \bs{\pi}_{-}^{(b)}(x)}\bs{1} \ud x & = 1.
\end{align} 
\end{cor}

\begin{proof}
The proof is in two steps.  Firstly, we show that the right-hand side
of (\ref{e:pms}) is a solution of the system (\ref{eqn:pms-1}).
Indeed,
\begin{align*}
\vect\nu G^{(b)} \vligne{-T_{++}^{-1} & 0 \\ 0 & -T_{--}^{-1}}  W & = 
 -\vect\nu G^{(b)} + \vect\nu G^{(b)} \vligne{-T_{++}^{-1} & 0 \\ 0 &
   -T_{--}^{-1}}   \vligne{0 & T_{+-} \\ T_{-+}  & 0}  G^{(b)}\\
 & = - \vect\nu G^{(b)}  + \vect\nu H G^{(b)}  \\
 & = 0
\end{align*}%
by definition of $\vect\nu$.   This proves (\ref{e:pms}), where $c$ is
some scaling constant.

Secondly, {the vector $\vect y$ defined in (\ref{e:y}) may be
  written as}   
\begin{align*}
  \vect y % &  = 
%   \vligne{\bs{p}_{+}^{(b)} & \bs{p}_{-}^{(0)}}
% 	\vligne{	0 & T_{+-} \\     T_{-+} & 0 	}
%     N^{-1}
% \\
 & = c \vligne{\bs{\nu}_{+} &
        \bs{\nu}_{-}} G^{(b)} \vligne{-T_{++}^{-1} & 0 \\ 0 &
        -T_{--}^{-1}} 
	\vligne{	0 & T_{+-} \\     T_{-+} & 0 	}
N^{-1}
 \\
& = c \vect \nu H N^{-1} \\
& = c \vect \nu N^{-1},
\end{align*}
which proves (\ref{e:piepsilon}).
\end{proof} 

%{
To prove Theorem~\ref{theo:sdmmbm} below, we analyse in a succession
of lemmas the behaviour of the factors in (\ref{e:piepsilon}) as
functions of $\epsi$.
%}

\begin{lem}
   \label{t:calg}
   The matrices $K_\epsi$, $K_\epsi^*$, $U_\epsi$, $U^*_\epsi$ and the
   inverse of $N$ are such that  

\begin{align}
   \label{eqn:K}
K_{\varepsilon} & = K_0 + O(\varepsilon)   \qquad \mbox{with \ }K_0 =
\Psi_1 \Theta^{-1} + 2V^{-1}D, \\
   \label{eqn:Ks}
K^*_{\varepsilon} & = K_0^* + O(\varepsilon) \qquad \mbox{with \ }K_0^* =
\Psi_1^* \Theta^{-1} - 2V^{-1}D, \\[0.25\baselineskip]
	\label{eqn:U}
U_{\varepsilon} & = \Theta^{-1}\Psi_1 + \varepsilon(\Theta^{-1}Q + V^{-1}D\Psi_1) + O(\varepsilon^2), \\
	\label{eqn:Us}
U^*_{\varepsilon} & = \Theta^{-1}\Psi_1^* + \varepsilon(\Theta^{-1}Q -
V^{-1}D\Psi_1) + O(\varepsilon^2)   
\end{align}
and
\begin{equation}
  \label{eqn:lastpie-2}
  {N}^{-1} = \vligne{
    I & \ue^{K_0 b} \\
    \ue^{K_0^* b} & I 
  }^{-1}  + O(\epsi).
\end{equation}
\end{lem}
\begin{proof}
The expressions for $K_\epsi$ and $K^*_\epsi$ are from
\cite[Lemma~3.6]{ln13} and a similar proof gives the expressions for
$U_\epsi$ and $U^*_\epsi$.   Equation (\ref{eqn:lastpie-2}) follows
from (\ref{eqn:Psie}, \ref{eqn:Psies}, \ref{eqn:K}, \ref{eqn:Ks}).
\end{proof}

%We refer the interested reader to the Appendix for a proof of Theorem~\ref{theo:finitefluid}. As the proof (should) shows, 

The matrices $\Lambda_{++}^{(b)}, \widetilde{\Lambda}_{--}^{(b)}, \Psi_{+-}^{(b)}, \widetilde{\Psi}_{-+}^{(b)}$ all have probabilistic interpretations. We write $\P_{(x,k,i)}[\cdot]$ as shorthand for $\P[\cdot | L(0) = x, \beta_{\lambda}(0) = k, \varphi_{\lambda}(0) = i]$ for $k \in \{1,2\}$ and $i \in \mathcal{M}$, and let $\tau_x = \inf \{t > 0: L(t) = x\}$ be the hitting time to level $x$, for $x \in [0,b]$. Then, for $i, j \in \mathcal{M}$ 
\begin{align*}
[\Lambda_{++}^{(b)}]_{(1,i),(1,j)} & = \P_{(0,1,i)}[\tau_b < \infty, \tau_b < \tau_0, \beta_{\lambda}(\tau_b) = 1, \varphi_{\lambda}(\tau_b) = j], \\
[\widetilde{\Lambda}_{--}^{(b)}]_{(2,i),(2,j)}  & = \P_{(b,2,i)}[\tau_0 < \infty, \tau_0 < \tau_b, \beta_{\lambda}(\tau_0) = 2, \varphi_{\lambda}(\tau_0) = j], \\
[\Psi_{+-}^{(b)}]_{(1,i),(2,j)} & =  \P_{(0,1,i)}[\tau_0 < \infty, \tau_0 < \tau_b, \beta_{\lambda}(\tau_0) = 2, \varphi_{\lambda}(\tau_0) = j], 
\\
[\widetilde{\Psi}_{-+}^{(b)}]_{(2,i),(1,j)} & =  \P_{(b,2,i)}[\tau_b < \infty, \tau_b < \tau_0, \beta_{\lambda}(\tau_b) = 1, \varphi_{\lambda}(\tau_b) = j].
\end{align*} 

\begin{lem} 
	\label{lem:quadmats}
%
%{
The matrix $G^{(b)}$ is given by 
\begin{align}
   \label{e:Pofepsilon}
G^{(b)} & = J + \varepsilon G_1^{(b)}  + O(\varepsilon^2)\\
\nonumber
\mbox{where} \quad
J & = \vligne{0 & I \\ I & 0  },
\qquad \qquad
G_1^{(b)}  = \vligne{L_1 & P_1 \\\widetilde{P}_1 & \widetilde{L}_1  }
 % \end{align}
% }
% 	%
% {with} 
% 	% 
% \begin{align} 
\\
		\label{eqn:L1}
	L_1 & = (\Psi_1 - P_1)\ue^{-\Theta^{-1}\Psi_1b}, \\
		\label{eqn:L1t}
	\widetilde{L}_1 & = (\Psi_1^* - \widetilde{P}_1)\ue^{-\Theta^{-1}\Psi_1^* b}, \\
		\label{eqn:P1} 
	P_1 & = (\Psi_1^* \ue^{\Theta^{-1}\Psi_1^*b}\ue^{\Theta^{-1}\Psi_1b} + \Psi_1)(I - \ue^{\Theta^{-1}\Psi_1^*b}\ue^{\Theta^{-1}\Psi_1b})^{-1}, \\
		\label{eqn:P1t} 
	\widetilde{P}_1 & = (\Psi_1 \ue^{\Theta^{-1}\Psi_1 b}\ue^{\Theta^{-1}\Psi_1^*b} + \Psi_1^*)(I - \ue^{\Theta^{-1}\Psi_1b}\ue^{\Theta^{-1}\Psi_1^*b})^{-1}. 
	\end{align} 
%}
\end{lem} 
\begin{proof} 
By Lemma~\ref{t:psi} and (\ref{eqn:U},  \ref{eqn:Us}),
\[
\Psi_{\varepsilon}\ue^{U_\varepsilon b} = \ue^{U_{\varepsilon}b} +
\varepsilon \Psi_1 \ue^{U_\varepsilon b} + O(\varepsilon^2) =
\ue^{\Theta^{-1} \Psi_1 b} + \varepsilon \Upsilon_1 +
O(\varepsilon^2),
\]
where $\varepsilon \Upsilon_1 \rightarrow 0$ as $\varepsilon
\rightarrow 0$, and 
\[
{\Psi}^*_{\varepsilon}\ue^{{U}^*_{\varepsilon}b}  =
\ue^{{U}^*_{\varepsilon}b} + \varepsilon
{\Psi}^*_1\ue^{U^*_{\varepsilon}b} + O(\varepsilon^2) =
\ue^{\Theta^{-1}{\Psi}^*_1 b} + \varepsilon {\Upsilon}^*_1 +
O(\varepsilon^2),
\] 
where ${\Upsilon}^*_1 \rightarrow 0$ as $\varepsilon \rightarrow 0$.
Then, we find from the system
\eqref{eqn:probmats}, that the
matrices $\Lambda_{++}^{(b)}$, $\widetilde{\Lambda}_{--}^{(b)},
\Psi_{+-}^{(b)},$ and $\widetilde{\Psi}_{-+}^{(b)}$ can be written as
\begin{align}
  \Lambda_{++}^{(b)} & = L_0 + \varepsilon L_1 + O(\varepsilon^2), \qquad \widetilde{\Lambda}_{--}^{(b)} = \widetilde{L}_0 + \varepsilon\widetilde{L}_1 + O(\varepsilon^2), \nonumber \\
  \Psi_{+-}^{(b)} & = P_0 + \varepsilon P_1 + O(\varepsilon^2), \qquad
  \widetilde{\Psi}_{-+}^{(b)} = \widetilde{P}_0 + \varepsilon
  \widetilde{P}_{1} + O(\varepsilon^2). \nonumber
\end{align} 
{This} leads to a new system of equations, the first of which is
\begin{align}
	\label{eqn:first}
        & L_0 + \varepsilon  L_1 + O(\varepsilon^2) + \{P_0 + \varepsilon P_1 + O(\varepsilon^2)\}\{\ue^{\Theta^{-1}{\Psi}^*_1b} + \varepsilon {\Upsilon}^*_1 + O(\varepsilon^2)\} \nonumber  \\
        &\quad = \ue^{U^*_{\varepsilon}b} \nonumber \\
        &\quad = \Psi^*_{\varepsilon}\ue^{U^*_{\varepsilon}b} - \{\varepsilon \Psi^*_1 + O(\varepsilon^2)\}\ue^{U_{\varepsilon}^*b} \nonumber \\
        & \quad = \ue^{\Theta^{-1}\Psi^*_{1}b}  \varepsilon \Psi_1^*\ue^{U_{\varepsilon}^*b} {+ \varepsilon \Upsilon_1^*} + O(\varepsilon^2),
\end{align} 
the second, similarly, is 
\begin{align} 
& \{\widetilde{P}_0 + \varepsilon \widetilde{P}_1 + O(\varepsilon^2)\}\{ \ue^{\Theta^{-1}\Psi_1 b} + \varepsilon \Upsilon_1 + O(\varepsilon^2)\} + \widetilde{L}_0 + \varepsilon \widetilde{L}_1 + O(\varepsilon^2) \nonumber \\
	\label{eqn:second}
& \quad = \ue^{\Theta^{-1} \Psi_1 b}  - \varepsilon \Psi_1
\ue^{U_{\varepsilon}b} +  {\varepsilon \Upsilon_1} + O(\varepsilon^2),
\end{align} 
and the third and fourth are 
\begin{align}
&\{L_0 + \varepsilon L_1 + O(\varepsilon^2)\}\{\ue^{\Theta^{-1}\Psi_1 b} + \varepsilon \Upsilon_1 + O(\varepsilon^2)\} + P_0 + \varepsilon P_1  + O(\varepsilon^2) \nonumber \\
	\label{eqn:third}
 & \quad = I + \varepsilon \Psi_1 + O(\varepsilon^2),  \\
& \widetilde{P}_0 + \varepsilon \widetilde{P}_1 + O(\varepsilon^2) + \{\widetilde{L}_0 + \varepsilon \widetilde{L}_1 + O(\varepsilon^2)\}\{\ue^{\Theta^{-1} \Psi_1^* b} + \varepsilon \widetilde{\Upsilon}_1 + O(\varepsilon^2)\}  \nonumber \\
	\label{eqn:fourth}
& \quad = I + \varepsilon \Psi_1^* + O(\varepsilon^2). 
\end{align} 
We match coefficients for $\varepsilon^0$ in both sides of \eqref{eqn:first}--\eqref{eqn:fourth} to obtain 
\begin{align}
	\label{eqn:0-1}
L_0 + P_0\ue^{\Theta^{-1} \Psi^*_1 b} & =  \ue^{\Theta^{-1} \Psi^*_1 b},  \\
	\label{eqn:0-4}
\widetilde{P}_0 \ue^{\Theta^{-1} \Psi_1 b} + \widetilde{L}_0 & = \ue^{\Theta^{-1} \Psi_1 b}, \\
	\label{eqn:0-2}
L_0\ue^{\Theta^{-1} \Psi_1 b} + P_0 & = I, \\
	\label{eqn:0-3}
\widetilde{P}_0 + \widetilde{L}_0\ue^{\Theta^{-1} \Psi_1^* b} & = I.   
\end{align}
Equations~\eqref{eqn:0-1} and~\eqref{eqn:0-2} imply that $L_0\{I - \ue^{\Theta^{-1}\Psi_1b}\ue^{\Theta^{-1}\Psi_1^*b}\} = 0$. By the proof of~\cite[Lemma~3.4]{ln13}, the matrices $\Theta^{-1}\Psi_1$ and $\Theta^{-1}\Psi_1^*$ are generators, one of which is for a transient Markov chain. Thus, $\ue^{\Theta^{-1}\Psi_1b}\ue^{\Theta^{-1}\Psi_1^*b}$ is sub-stochastic and $\{I -  \ue^{\Theta^{-1}\Psi_1b} \ue^{\Theta^{-1}\Psi_1^*b}\}$ is invertible, which give  $L_0 = 0$ and $P_0 = I$. Similarly,~\eqref{eqn:0-3} and \eqref{eqn:0-4} imply $\widetilde{L}_0 = 0$ and $\widetilde{P}_0 = I$. 

Next, we {equate} coefficients for $\varepsilon$ in both sides of \eqref{eqn:first}--\eqref{eqn:fourth} to obtain 
\begin{align}
	\label{eqn:1-1} 
L_1 + P_1 \ue^{\Theta^{-1} \Psi_1^* b}  & =  - \Psi_1^* \ue^{\Theta^{-1}\Psi_1^* b}, \\
	\label{eqn:1-4} 
\widetilde{P}_1\ue^{\Theta^{-1}\Psi_1b} + \widetilde{L}_1 & =  - \Psi_1\ue^{\Theta^{-1}\Psi_1b}, \\
	\label{eqn:1-2} 
L_1 \ue^{\Theta^{-1} \Psi_1 b} + P_1 & = \Psi_1, \\  
	\label{eqn:1-3} 
\widetilde{P}_1 + \widetilde{L}_1 \ue^{\Theta^{-1} \Psi_1^* b} & = \Psi_1^*.
\end{align} 
As $\ue^{\Theta^{-1} \Psi_1 b}$ and $\ue^{\Theta^{-1}\Psi_1^*b}$ are invertible, constraints \eqref{eqn:1-2} and \eqref{eqn:1-3} prove~\eqref{eqn:L1} and~\eqref{eqn:L1t}, respectively, whereas \eqref{eqn:P1} follow from \eqref{eqn:L1} and constraint \eqref{eqn:1-1}, and \eqref{eqn:P1t} follow from \eqref{eqn:L1t} and constraint \eqref{eqn:1-4}.
\end{proof} 

By Lemma \ref{lem:quadmats}, $\lim_{\varepsilon \rightarrow 0} G^{(b)} = J,$ which
makes sense probabilistically. As $\varepsilon \rightarrow 0$,
$G^{(b)}$ converges to the matrix of hitting and returning
probabilities in the limiting Markov-modulated Brownian motion
$\{{Y}(t), \kappa(t)\}$. Thus, for example,
$\lim_{\varepsilon \rightarrow 0} \Lambda_{++}^{(b)} = 0$ implies that
the conditional probabilities of $Y(t)$ hitting the upper boundary $b$
before returning to the initial level $0$ are zero.

\begin{lem}
   \label{t:nu}
The vector $\vect\nu$ is such that $\vect\nu = \vect\nu_0 +
O(\varepsilon)$, where $\vect\nu_0$ is the unique probability
vector, solution of the system $\vect\nu_0  G_1^{(b)} = \vzero$,
$\vect\nu_0 \vone = 1$.
\end{lem}
\begin{proof}
We readily observe from the definition of the matrix $T$ that 
\[
H = G^{(b)} \vligne{0 & I + \varepsilon^2 Q + O(\varepsilon^4) \\
   I + \varepsilon^2 Q + O(\varepsilon^4) & 0} 
= I + \varepsilon G_1^{(b)} J + O(\varepsilon^2)
\]
so that $\vect\nu$ is of the form $\vect\nu = \vect\nu_0 + \varepsilon
\vect\nu_1 + O(\varepsilon^2)$.  If we equate the coefficients of
equal power of $\epsi$ on both sides of $\vect\nu = \vect\nu H$,
$\vect\nu\vone = 1$, we find that 
\[
\vect\nu_0 = \vect\nu_0, \qquad \vect\nu_1 = \vect\nu_1+\vect\nu_0
G_1^{(b)} J, \qquad \vect\nu_0 \vone = 1,
\]
or
\[
\vect\nu_0 G_1^{(b)} J = \vzero, \qquad \vect\nu_0 \vone = 1.
\]
Since $J$ is nonsingular and $J^2 = I$, we may rewrite the system above as
\begin{equation}
   \label{e:nuzero}
\vect\nu_0 G_1^{(b)} = \vect\nu_0 J J G_1^{(b)} = \vzero, \qquad
\vect\nu_0 \vone = \vect\nu_0 J \vone = 1.
\end{equation}
Now, the matrix 
\[
J G_1^{(b)} = \vligne{\widetilde P_1 & \widetilde L_1 \\ L_1 & P_1}
\] 
is an irreducible generator, as we show
below, and this entails that the system $\vect x J G_1^{(b)} = \vzero$,
$\vect x \vone = 1$ has a unique solution, so that the lemma will be proved.

Since $G^{(b)}$ is stochastic, we conclude from (\ref{e:Pofepsilon})
that $L_1$ and $\widetilde L_1$ are both nonnegative, as well as all
the off-diagonal elements of $P_1$ and of $\widetilde P_1$.  As
$G^{(b)} \vone = \vone$, $G_1^{(b)} \vone = \vzero$ and the diagonal
elements of $P_1$ and of $\widetilde P_1$ must be less than, or equal
to zero.  Finally, as $G^{(b)}$ is irreducible, all diagonal elements
of $P_1$ and $\widetilde P_1$ must be strictly negative.
\end{proof}

\begin{lem}
   \label{t:cplus}
The last factor in (\ref{e:piepsilon}) is
\begin{align} 
\left[
\begin{array}{rr} 
	C_{+}^{-1} & \Psi_{\varepsilon} |C_{-}|^{-1}\\
	 \Psi^*_{\varepsilon} C_{+}^{-1} & |C_{-}|^{-1}
	 \end{array}\right] = \epsi
       \vligne{\Theta^{-1} & \Theta^{-1} \\ \Theta^{-1} & \Theta^{-1}} + O(\epsi^2).
\end{align} 
\end{lem}
\begin{proof}
By the definition of $C$, we have 
\[
C_+^{-1} = (1/\epsi \Theta + D)^{-1} = \epsi \Theta^{-1} + O(\epsi^2)
\]
and similarly $C_-^{-1} = -\epsi \Theta^{-1} + O(\epsi^2)$.  To
conclude the proof, we use (\ref{eqn:Psie}, \ref{eqn:Psies}).
\end{proof}

\begin{lem}
   \label{t:c}
The normalizing constant $c$ in \eqref{e:piepsilon} is of the form
\begin{equation}
   \label{e:c}
c = \epsi^{-1} c_{-1} + c_0 + O(\epsi).
\end{equation}
\end{lem}
\begin{proof}
The exact expression of $c_{-1}$ is not as  important as the form of the
right-hand side in~(\ref{e:c}) and we shall omit the details in the
argument below.   By (\ref{eqn:pms-2b}),
\begin{align}
   \label{e:cinv}
c^{-1} & = \vect\nu G^{(b)} \vligne{-T_{++}^{-1} & 0 \\ 0 &
  -T_{--}^{-1}}\vone   \\
  \nonumber
& \quad + \vect\nu {N}^{-1} \int_0^b \vligne{ 
	\ue^{K_{\varepsilon}x} & 0  \\
	0 & \ue^{{K}^*_{\varepsilon}(b - x)} } \ud x
	\left[\begin{array}{rr} 
	C_{+}^{-1} & \Psi_{\varepsilon} |C_{-}|^{-1}\\
	 \Psi^*_{\varepsilon} C_{+}^{-1} & |C_{-}|^{-1}\end{array}\right]
        \vone
\end{align}
By Lemmas~\ref{t:nu}, \ref{t:cplus} and \ref{t:calg}, the first term
in (\ref{e:cinv}) is $O(\epsi^2)$ and the second is $O(\epsi)$, thus
$c^{-1}$ is $O(\epsi)$ and this justifies (\ref{e:c}).
\end{proof}

We may now bring together all our partial results.

\begin{theo}
	\label{theo:sdmmbm} 
	The stationary density of the two-sided Markov-modulated Brownian motion $\{\widetilde{Y}(t), \kappa(t)\}$ is given by 
\begin{align}
	\label{eqn:pimmbm}
\lim_{\varepsilon \rightarrow 0} \bs{\pi}^{(b)}_{\varepsilon}(x)(\bs{1}_2 \otimes I_m) 
& =
 c^*  \vect\nu_0
%\vligne{\bs{p}_{-1}^{(0)} & \bs{p}_{+1}^{(b)}}
\vligne{
I & \ue^{K_0 b} \\
\ue^{K_0^* b} & I 
}^{-1}
\left[\begin{array}{rr}
\ue^{K_0 x}\Theta^{-1} \\
\ue^{K_0^*(b - x)}\Theta^{-1} 
\end{array}\right],
\end{align} 	
for $x \in (0, b)$, where $\vect\nu_0$ is the unique probability
vector that is solution of the system $\vect\nu_0 G_1^{(b)} = \vzero$,
$\vect\nu_0 \vone = 1$, and $c^*$ is a normalizing constant.

	The probability masses of $\{\widetilde{Y}(t), \kappa(t)\}$ at the two boundaries are zero. 
\end{theo} 
\begin{proof}
This is a direct consequence of Lemmas \ref{lem:quadmats} to \ref{t:c}.
\end{proof}

\begin{rem}   \rm
Theorem \ref{theo:sdmmbm} shows that the stationary density is made up
of two components: the factor $\bs{\nu}_0$ is about the limiting
behaviour of $\{\widetilde{Y}(t), \kappa(t)\}$ at the boundaries, the
matrix product is about its limiting behaviour in the interior
$(0,b)$. This factorization implies that to modify the boundary behaviour of  $\{\widetilde{Y}(t), \kappa(t)\}$ would affect the vector $\bs{\nu}_0$ only. 
\end{rem}

\section{Comparison with existing literature} 
	\label{sec:comp} 
%	\intrusion{Need to compare our Theorem~\ref{theo:sdmmbm} to the results of Rogers~\cite{roger94} and of Ivanovs~\cite{ivanovs10}}
%%%Rogers~\cite{roger94} derived the stationary density for the process that is an independent sum of a Markov-modulated linear drift and a standard Brownian motion. This model, as reasoned in Section~3.2 of Ivanovs~\cite{ivanovs10}, is equivalent in distribution to an MMBM with all variances being positive. This assumption is the same as what we have assumed throughout this paper (see Assumption~\ref{ass:mmbm}), and is relaxed in Ivanovs~\cite{ivanovs10}.
%%%
%%%The analysis in~\cite[Section~3.2]{ivanovs10} concluded that, up to
%%%an error in sign in Rogers~\cite{roger94} and under the assumption
%%%that all variances are positive, the two papers obtained the same
%%%following result 
%
        Section~3.2 of Ivanovs~\cite{ivanovs10} shows that, under
        {assumption} of all
        variances being positive, both~\cite{ivanovs10}
        and~\cite{roger94} obtained the same stationary density of the
        two-sided Markov-modulated Brownian motion $\widetilde{Y}(t)$
        {\em conditioned} on the phase $\kappa(\cdot)$:
\begin{align}
[\bs{f}(x)]^{\top} & =  \lim_{t \rightarrow \infty}  \frac{\ud}{\ud x} \P[\widetilde{Y}(t) \leq x | \kappa(t)], \nonumber \\
                         & = -\{\ue^{x\overline{\Omega}_{+}}\overline{\Omega}_{+} + \ue^{(b - x)\overline{\Omega}_{-}}\overline{\Omega}_{-}\ue^{b\overline{\Omega}_{+}}\}(I - \ue^{b\overline{\Omega}_{-}}\ue^{b\overline{\Omega}_{+}})^{-1}\bs{1},
                         	\label{eqn:ri}
\end{align} 
where $\overline{\Omega}_{+}$ and $\overline{\Omega}_{-}$ are
{respectively} the generators of first passage times to level $x$ and {level} $-x$ in
$\{\overline{Y}(t), \overline{\kappa}(t): t \geq 0\}$, the
time-reversed version of the unbounded MMBM $\{Y(t),
\kappa(t)\}$. {Each is} a solution to
one of the two matrix quadratic equations
\begin{align}
\frac{1}{2}V X^2 \mp D X + \Delta_{1/\bs{\alpha}}Q^{\top}\Delta_{\bs{\alpha}} = 0.
\end{align} 
The proof of {Theorem~3.7} in \cite{ln13} gives a relationship between $K_0$ and $\overline{\Omega}_{+}$:
\begin{align}
\overline{\Omega}_{+}^{\top} =  \Delta_{\bs{\alpha}}\Theta K_0 \Theta^{-1}\Delta_{1/\bs{\alpha}},  \nonumber
\end{align} 
and, similarly, we also have $ \overline{\Omega}_{-}^{\top} =  \Delta_{\bs{\alpha}}\Theta K_0^* \Theta^{-1}\Delta_{1/\bs{\alpha}}$. Thus, the conditional stationary density~\eqref{eqn:ri} can be rewritten as 
\begin{align}
\bs{f}(x) & = -\bs{1}^{\top}(I - \ue^{b \overline{\Omega}_{+}^{\top}}\ue^{b\overline{\Omega}_{-}^{\top}})^{-1}\{\overline{\Omega}_{+}^{\top}\ue^{x \overline{\Omega}_{+}^{\top}} + \ue^{b\overline{\Omega}_{+}^{\top}}\overline{\Omega}_{-}^{\top}\ue^{(b - x)\overline{\Omega}_{-}^{\top}}\} \nonumber \\
%
%& = -\bs{1}^{\top}\Delta_{\bs{\alpha}}\Theta(I -\ue^{bK_0}\ue^{bK_0^*})^{-1}\Theta^{-1}\Delta_{1/\bs{\alpha}} \times \nonumber \\
%& \hspace*{0.7cm} \{\Delta_{\bs{\alpha}}\Theta K_0 \ue^{x K_0}\Theta^{-1}\Delta_{1/\bs{\alpha}} + \Delta_{\bs{\alpha}}\Theta \ue^{b K_0} K_0^* \ue^{(b - x)K_0^*}\Theta^{-1}\Delta_{1/\bs{\alpha}}\} \nonumber \\
& = -\bs{\alpha}\Theta(I - \ue^{bK_0} \ue^{bK_0^*})^{-1}(K_0 \ue^{x K_0} + \ue^{bK_0} K_0^*\ue^{(b - x)K_0^*})\Theta^{-1}\Delta_{1/\bs{\alpha}}, \nonumber
\end{align} 
and thus the joint stationary density for $\{\widetilde{Y}(t), \kappa(t)\}$ is given by 
\begin{align}
  \nonumber	%\label{eqn:joint-Ivanovs}
\bs{f}(x)\Delta_{\bs{\alpha}} & = -\bs{\alpha}\Theta(I - \ue^{bK_0} \ue^{bK_0^*})^{-1}(K_0 \ue^{x K_0} + \ue^{bK_0} K_0^*\ue^{(b - x)K_0^*})\Theta^{-1} \nonumber \\
& = -\bs{\alpha}\Theta(I - \ue^{b K_0}\ue^{b K_0^*})^{-1}\vligne{K_0 & \ue^{b K_0}K_0^*}\left[\begin{array}{rr}\ue^{xK_0} \Theta^{-1} \\ \ue^{(b - x)K_0^*}\Theta^{-1}\end{array}\right],
\end{align} 
%%
%%Let us now turn the attention to our representation \eqref{eqn:pimmbm}. First, 
%%% 
%%\begin{align}
%%& \vligne{
%%I & \ue^{K_0 b} \\
%%\ue^{K_0^*b} & I 
%%}^{-1} \nonumber  = 
%%\vligne{
%%(I - \ue^{K_0 b} \ue^{K_0^*b})^{-1} & -(I - \ue^{K_0 b}\ue^{K_0^*b})^{-1} \ue^{K_0 b} \\ 
%%\vspace*{-0.3cm} \\
%%-\ue^{K_0^*b}(I - \ue^{K_0b} \ue^{K_0^*b})^{-1} & I + \ue^{K_0^*b}(I - \ue^{K_0b}\ue^{K_0^*b})^{-1} \ue^{K_0b}
%%}.
%%\end{align} 
%%
which coincides with our \eqref{eqn:pimmbm}
{if} % 
\begin{align} 
  \nonumber   %	\label{eqn:a}
-\bs{\alpha}\Theta\vligne{I - \ue^{bK_0}\ue^{bK_0^*}}^{-1}\vligne{K_0
  & \ue^{bK_0}K_0^*} = c^* \vect\nu_0 \vligne{I & \ue^{K_0 b} \\ \ue^{K_0^*b} & I}^{-1}.
\end{align} 
{This is} shown through tedious algebraic manipulations. 

\bibliographystyle{abbrv}
\bibliography{two-sided}

\end{document}